\documentclass[12pt,a4paper]{article}
\usepackage{graphicx}
\usepackage{amsmath,amsthm,amsfonts} 
\usepackage{txfonts,bm}

    \DeclareMathOperator\fix{f\/ix}
    
    \DeclareMathOperator\id{id}
    
    \DeclareMathOperator\GL{GL}
    \DeclareMathOperator\diag{diag}

    \DeclareMathOperator\Trans{{T}}
    
    \DeclareMathOperator\dis{\triangle}
    \DeclareMathOperator\spn{span}
    \DeclareMathOperator\rank{rank}
    \DeclareMathOperator\Gr{Gr}

\newcommand{\bC}{{\mathbb C}} 
\newcommand{\bP}{{\mathbb P}}
\newcommand{\bR}{{\mathbb R}}

\newcommand{\spmatrix}[1]
{\mbox{\scriptsize\setlength\arraycolsep{0.5\arraycolsep}$\begin{pmatrix}#1\end{pmatrix}$}}

\newtheorem{lem}{Lemma}
\newtheorem{thm}{Theorem}
\newtheorem*{conv}{Convention}
{\theoremstyle{definition}

} {\theoremstyle{remark}
\newtheorem{rem}{Remark}
}

\begin{document}

\title{Projective Lines over Jordan Systems and Geometry of Hermitian Matrices}

\author{Andrea Blunck \and Hans Havlicek}
\date{}

\maketitle

\centerline{\emph{Dedicated to Mario Marchi on the occasion of his 70th
birthday}}

\vspace{12pt}

\begin{abstract}
Any set of $\sigma$-Hermitian matrices of size $n\times n$ over a field with
involution $\sigma$ gives rise to a \emph{projective line\/} in the sense of
ring geometry and a \emph{projective space\/} in the sense of matrix geometry.
It is shown that the two concepts are based upon the same set of points, up to
some notational differences.
\end{abstract}

\section{Introduction}\label{se:intro}

Let $R=K^{n\times n}$ be the ring of $n\times n$ matrices over a (not
necessarily commutative) field $K$ which admits an involution $\sigma$. We
denote by $H_\sigma\subset R$ the subset of $\sigma$-Hermitian matrices. We
exhibit two well known constructions: The \emph{projective line over the Jordan
system} $H_\sigma\subset R$ is a subset of the point set of the
\emph{projective line over the matrix ring} $R$. It comprises all points which
can be written in the form $R(T_2T_1-I,T_2)$ with $\sigma$-Hermitian matrices
$T_1, T_2$; cf.\ \cite[3.1.14]{blunck+he-05}. The \emph{projective space of
$\sigma$-Hermitian $n\times n$ matrices\/} is a subset of the point set of the
\emph{projective space of $n\times n$ matrices over} $K$. Its points are the
left row spaces of those matrices $(A,B)\in K^{n\times 2n}$ which have full
left row rank and are composed of blocks $A,B\in K^{n\times n}$ satisfying
$A(B^\sigma)^{\Trans}=B(A^\sigma)^{\Trans}$; cf.\ \cite[6.8]{wan-96}.
\par
We recall in Section~\ref{se:square} that the point set of the projective line
over $R$ is, up to a natural identification, the Grassmannian of
$n$-dimensional subspaces of $K^{2n}$ which in turn is nothing but the point
set of the projective space of $n\times n$ matrices over $K$. In
Section~\ref{se:herm} we exhibit the two subsets which arise from
$\sigma$-Hermitian matrices according to the above mentioned constructions. The
coincidence of these two subsets is not obvious. Indeed, in the ring-geometric
setting we get a set of points in terms of a \emph{parametric representation},
whereas in the matrix-geometric setting there is a \emph{matrix equation\/}
which has to be satisfied. Our main result (Theorem~\ref{thm:1}) states that
the two subsets coincide. The proof of one inclusion simply amounts to plugging
in the parametrisation in the matrix equation. Our proof of the other inclusion
is more involved. It uses the rather technical Lemma~\ref{lem:1} and
Lemma~\ref{lem:2}, which is geometric in flavour, as it deals with maximal
totally isotropic subspaces of a $\sigma$-anti-Hermitian sesquilinear form. For
a commutative field $K$ and $\sigma=\id_K$ our Lemma~\ref{lem:2} turns into the
result \cite[Satz~10.2.3]{blunck+he-05}.
\par
As an application we show in Remarks~\ref{rem:1}--\ref{rem:5} how several
results from ring geometry can be translated to projective matrix spaces.

\section{Square matrices}\label{se:square}

Let $K$ be any (not necessarily commutative) field and $n\geq 1$. We shall be
concerned with the ring $R:= K^{n\times n}$ of $n\times n$ matrices with
entries in $K$. Any $r\times s$ matrix over $R$ can be viewed as an $rn\times
sn$ matrix over $K$ which is partitioned into $rs$ blocks of size $n\times n$
and vice versa. An $r\times r$ matrix over $R$ is invertible if, and only if,
it is invertible as an $rn\times rn$ matrix over $K$.

Consider the free left $R$-module $R^2$ and the group $\GL_2(R)=\GL_{2n}(K)$ of
invertible $2\times 2$-matrices with entries in $R$. A pair $(A,B)\in R^2$ is
called \emph{admissible}, if there exists a matrix in $\GL_2(R)$ with $(A,B)$
being its first row. Following \cite[p.~785]{herz-95} and \cite{blunck+he-05},
the {\em projective line over\/} $R$ is the orbit of the free cyclic submodule
$R(I,0)$ under the natural right action of $\GL_2(R)$, where $I$ and $0$ denote
the $n\times n $ identity and $n\times n$ zero matrix over $K$, respectively.
So
\begin{equation}\label{eq:P(R)}
  \bP(R):=R(I,0)^{\GL_2(R)}
\end{equation}
or, in other words, $\bP(R)$ is the set of all $p\subset R^2$ such that
$p=R(A,B)$ for an admissible pair $(A,B)\in R^2$. Two admissible pairs
represent the same point precisely when they are left-proportional by a unit in
$R$, i.~e., a matrix from $\GL_n(K)$. Conversely, if for some pair $(A',B')\in
R^2$ and an admissible pair $(A,B)\in R^2$ we have $R(A',B')=R(A,B)$ then
$(A',B')$ is admissible too. This follows from
\cite[Proposition~2.2]{blunck+h-00b}, because in $R$ the notions of ``right
invertibility'' and ``invertibility'' coincide.
\par
The projective line over $R$ allows the following description which makes use
of the \emph{left row rank\/} of a matrix $X$ over $K$ (in symbols: $\rank X$):
\begin{equation}\label{eq:P-Knn}
    \bP(R) = \{R(A,B)\mid A,B\in R,\; \rank (A,B)=n\}.
\end{equation}
Here $(A,B)$ has to be interpreted as the matrix arising from $A$ and $B$ by
means of horizontal augmentation. By (\ref{eq:P-Knn}), the point set of
$\bP(R)$ is in bijective correspondence with the Grassmannian $\Gr_{2n,n}(K)$
of $n$-dimensional subspaces of $K^{2n}$ via
\begin{equation}\label{eq:via}
    \bP(R)\to\Gr_{2n,n}(K) : R(A,B) \mapsto \mbox{left row space of~} (A,B).
\end{equation}
See \cite{blunck-99}, \cite{blunck+h-00b}, and \cite{thas-71} for this result
and its generalisations.
\begin{conv}
We do not distinguish between a point of\/ the projective line $\bP(R)$ and its
corresponding subspace of $K^{2n}$ via \emph{(\ref{eq:via})}.
\end{conv}

Following \cite{blunck+he-05} a ring will be called \emph{stable\/} if it has
stable rank $2$. By \cite[2.6]{veld-85}, our matrix ring $R=K^{n\times n}$ is
stable. This means that for each $(A,B)\in R^2$ which is \emph{unimodular},
i.~e., there are $X,Y\in R$ with $AX+BY=I$, there exists a matrix $W\in R$ such
that $A+BW\in\GL_n(K)$. See \cite[\S~2]{veld-85}. Due to stableness two
important results hold: Firstly, any unimodular pair $(A,B)\in R^2$ generates a
point \cite[2.11]{veld-85}. Secondly, \emph{Bartolone's parametrisation\/}
\begin{equation}\label{eq:bartolone}
    R^2 \to \bP(R) : (T_1,T_2)\mapsto R(T_2T_1 - I, T_2)
\end{equation}
is well defined and surjective \cite{bart-89}. It allows us to write the
projective line $\bP(R)$ in the form
\begin{equation}\label{eq:t2t1}
    \bP(R) = \{R(T_2T_1 - I, T_2)\mid T_1,T_2\in R\}.
\end{equation}
Formula (\ref{eq:t2t1}) was put in a more general context in
\cite{blunck+h-03}, and we shall follow the notation from there. Altogether we
have four equivalent descriptions of the projective line over a matrix ring
$R=K^{n\times n}$.

\par
The point set $\bP(R)$ is endowed with the symmetric and anti-reflexive
relation \emph{distant\/} ($\dis$) defined by
\begin{equation*}
  \dis:=\big(R(I,0), R(0,I)\big)^{\GL_2(R)}.
\end{equation*}
For arbitrary points $p=R(A,B)$ and $q= R(C,D)$ of $\bP(R)$ we obtain
\begin{equation*}
  p\,\dis\, q\,\Leftrightarrow\,
  \begin{pmatrix}A&B\\C&D\end{pmatrix}
  \in \GL_2(R)=\GL_{2n}(K).
\end{equation*}
This in turn is equivalent to the complementarity of the $n$-dimensional
subspaces of $K^{2n}$ which correspond to $p$ and $q$ via (\ref{eq:via}). The
vertices of the \emph{distant graph\/} on $\bP(R)$ are the points of $\bP(R)$,
two vertices of this graph are joined by an edge if, and only if, they are
distant. A crucial property of the projective line over our ring $R$, and more
generally over any stable ring, is as follows \cite[1.4.2]{herz-95}: Given any
two points $p$ and $q$ there exists some point $r$ such that $p\dis r\dis q$.
This implies that the distant graph on $\bP(R)$ is connected and that its
diameter is $\leq 2$.
\par
For example, given $p=R(I,0)$ and any other point $q\in\bP(R)$ we have
$q=R(T_2T_1 - I, T_2)$ by (\ref{eq:t2t1}). Then $r:=R(T_1,I)$ has the required
property
\begin{equation}\label{eq:distkette}
    R(I,0) \dis R(T_1,I)\dis R(T_2T_1 - I, T_2).
\end{equation}

\par
Comparing the description of the point set $\bP(K^{n\times n})=\bP(R)$ in
(\ref{eq:P-Knn}) with the definition of the point set of the \emph{projective
space of $m\times n$ matrices over} $K$ in \cite[3.6]{wan-96} one sees
immediately that the two definitions coincide for $m=n\geq 2$ (due to our
convention from above). So in our setting proper rectangular matrices as in
\cite{wan-96} are not allowed and, to be compatible with \cite{wan-96}, we
assume from now on that $n\geq 2$, i.~e., we disregard the projective line
$\bP(K^{1\times 1})$. There is an immaterial difference though, as we make use
of the vector space $K^{2n}$ rather than the projective space on $K^{2n}$ as in
\cite{wan-96}. This is only done in order to simplify notation.
\par
The major difference in the two approaches concerns the \emph{additional
structure\/} which is imposed: In the ring-theoretic setting this is the notion
of \emph{distance}, whereas in the matrix-theoretic setting the concept of
\emph{adjacency\/} ($\sim$) is used. Recall that two $n$-dimensional subspaces
of $K^{2n}$ are called adjacent if, and only if, their intersection has
dimension $n-1$. The vertices of the \emph{Grassmann graph\/} on
$\Gr_{2n,n}(K)$ are the elements of $\Gr_{2n,n}(K)$, two vertices are joined by
an edge if, and only if, they are adjacent. The graph-theoretical distance
between two vertices $W_1,W_2$ of the Grassmann graph on $\Gr_{2n,n}(K)$ equals
their \emph{arithmetical distance\/} $\dim(W_1+W_2)-m$
\cite[Proposition~3.32]{wan-96}.

\par
However, also the structural approaches can be shown to be equivalent, because
adjacency can be expressed in terms of being distant and vice versa: Two points
of $\bP(R)$ are distant if, and only if, they are at arithmetical distance $n$
in the Grassmann graph. The description of $\sim$ in terms of $\dis$ is more
subtle, and we refer to \cite[Theorem~3.2]{blunck+h-05a} for further details.
\par
Even though Bartolone's parametrisation (\ref{eq:t2t1}) has its origin in ring
geometry, the identification from (\ref{eq:via}) allows its reinterpretation as
a surjective parametric representation of $\Gr_{2n,n}(K)$ in the form
\begin{equation}\label{eq:allgemein}
     R^2\to \Gr_{2n,n}(K) : (T_1,T_2) \mapsto \mbox{left row space of~}(T_2T_1-I,T_2).
\end{equation}
We sketch some applications of this result in the geometry of square matrices:

\begin{rem}\label{rem:1}
The mapping (\ref{eq:allgemein}) has the disadvantage of being non-injective,
    but by choosing a fixed matrix $T_1^{(0)}$ an injective mapping $R\to
    \Gr_{2n,n}(K)$ is obtained from (\ref{eq:allgemein}). This mapping is
    easily seen to be an embedding of the matrix space $R= K^{n\times n}$ in
    the projective matrix space $\Gr_{2n,n}(K)$. For $T_1^{(0)}=0$ one gets the
    ``usual'' embedding, like in \cite{wan-96}.

\end{rem}

\begin{rem}\label{rem:2}
The left row spaces of the matrices $(I,0)$ and $(T_2T_1-I,T_2)$ have
arithmetical distance $k$ if, and only if, $\rank T_2=k$. Thus a
parametrisation of the ``spheres'' of $\Gr_{2n,n}(K)$ with ``centre'' $(I,0)$
and ``radius'' (arithmetical distance) $k$ can be obtained from
(\ref{eq:allgemein}) by imposing the extra condition $\rank T_2=k$, while
$T_1\in R$ is arbitrary. In particular, the case $k=1$ can be treated by
restricting the choice of $T_2$ to matrices of the form $c^{\Trans}\cdot d$
with $c, d \in K^{n}\setminus\{0\}$.
\par
Similarly, we may also parametrise any maximal set of mutually adjacent
elements and any pencil of $\Gr_{2n,n}(K)$ containing the left row space of
$(I,0)$ as follows. Firstly, let $T_2:=c^{(0)}{}^{\Trans}\cdot d$ for a fixed
vector $c^{(0)}\in K^{2n}\setminus\{0\}$ and a variable vector $d\in K^{2n}$.
Then (\ref{eq:allgemein}) gives the set of all $n$-dimensional subspaces which
contain the $(n-1)$-dimensional subspace given by the linear system $ \sum
_{i=1}^n x_i c_i^{(0)} = x_{n+1}= \cdots = x_{2n}=0$, where the $c_i^{(0)}$s
are the coordinates of $c^{(0)}$. Secondly, let $T_2:=c^{\Trans}\cdot d^{(0)}$
for a fixed vector $d^{(0)}\in K^{2n}\setminus\{0\}$ and a variable vector
$c\in K^{2n}$. Then (\ref{eq:allgemein}) gives the set of all $n$-dimensional
subspaces which are contained in the $(n+1)$-dimensional row space of the
matrix $\spmatrix{I&0\\0&d^{(0)}}$. Thirdly, let $T_2:=c^{(0)}{}^{\Trans}\cdot
t\cdot d^{(0)}$ for fixed vectors $c^{(0)}, d^{(0)}\in K^{2n}\setminus\{0\}$
and a variable $t\in K$. Then (\ref{eq:allgemein}) gives a pencil of
$n$-dimensional subspaces or, in the terminology of
\cite[Definition~3.11]{wan-96}, a \emph{line\/} of our projective matrix space.
\end{rem}

\begin{rem}\label{rem:3}
Let $\widehat{K^{2n}}$ be the $2n$-dimensional right column space over $K$. It
is the dual of $K^{2n}$. For each $n$-dimensional subspace $W$ of $K^{2n}$ the
linear forms (column vectors) which vanish on $W$ constitute the
$n$-dimensional annihilating subspace $W^\circ\subset \widehat{K^{2n}}$. We may
assume that $W$ is the left row space of $(T_2T_1-I,T_2)$ with $T_1,T_2\in R$.
Then
\begin{equation}\label{eq:ann}
    W^\circ =\mbox{right column space of~}\begin{pmatrix}-T_2\\T_1T_2-I\end{pmatrix}.
\end{equation}
For $T_1=0$ this result is folklore.
\end{rem}

\begin{rem}\label{rem:4}
Let $\iota : R\to R$ be any \emph{Jordan isomorphism\/} (see
\cite[9.1]{herz-95} or \cite[Definition~3.7]{wan-96}, where the term
\emph{semi-isomorphism\/} is used instead). Then the mapping
$\Gr_{2n,n}(K)\to\Gr_{2n,n}(K)$ given by
\begin{equation}\label{eq:jordan}
    \mbox{left row space of~}(T_2T_1-I,T_2)
    \mapsto \mbox{left row space of~}(T_2^\iota T_1^\iota - I, T_2^\iota)
\end{equation}
is well-defined. This follows from \cite[Theorem~2.4]{bart-89} or
\cite[Satz~4.2.11]{blunck+he-05} by removing superfluous conditions about the
ground field; see also \cite[Theorem~9.1.1]{herz-95}. The well-definedness is
also a direct consequence of \cite[Theorem~4.4]{blunck+h-03}. In our setting
there is an easier proof: The Jordan isomorphism $\iota$ is either of the form
$X \mapsto Q^{-1} X^\gamma Q$, with $\gamma$ an automorphism of $K$ and
$Q\in\GL_n(K)$, or of the form $X \mapsto Q^{-1}(X^\delta)^{\Trans} Q$, with
$\delta$ an antiautomorphism of $K$ and $Q\in\GL_n(K)$. See, e.~g.,
\cite[Theorem~3.24]{wan-96}. In the first case (\ref{eq:jordan}) coincides with
the natural action of the semilinear bijection $K^{2n}\to K^{2n}:x\mapsto
x^{\gamma}\cdot \diag(Q,Q)$ on $\Gr_{2n,n}(K)$. In the second case we consider
the non-degenerate sesquilinear form
\begin{equation*}
    K^{2n}\times K^{2n}\to K : (x,y)\mapsto x\cdot
    \begin{pmatrix}0 & -Q^{-1}\\ Q^{-1} & \phantom{-}0 \end{pmatrix}\cdot (y^\delta)^{\Trans}.
\end{equation*}
It acts on $\Gr_{2n,n}(K)$ by sending $W\in\Gr_{2n,n}(K)$ to its perpendicular
subspace $W^\perp$ \cite[Proposition~3.42]{wan-96} and has the required
properties, since
\begin{equation*}
    (T_2^\iota T_1^\iota - I, T_2^\iota)\cdot
    \begin{pmatrix}0 & -Q^{-1}\\ Q^{-1} & 0 \end{pmatrix}\cdot
    \left((T_2T_1-I,T_2)^\delta\right)^{\Trans} =0.
\end{equation*}
The previous formulas show that (\ref{eq:jordan}) together with the natural
action of $\GL_2(R)=\GL_{2n}(K)$ provides \emph{a unified explicit description
of adjacency preserving transformations of\/} $\Gr_{2n,n}(K)$ which avoids the
distinction (like, e.~g., in \cite[Theorem~3.45]{wan-96}) between semilinear
bijections and non-degenerate sesquilinear forms.
\end{rem}

\section{$\sigma$-Hermitian matrices}\label{se:herm}

Suppose now that the field $K$ admits an antiautomorphism $\sigma$ such that
$\sigma^2=\id_K$. Such a mapping will be called an \emph{involution}. Observe
that we do not adopt any of the extra assumptions on $\sigma$ from
\cite[p.~306]{wan-96}. As before, we let $R=K^{n\times n}$ and the identity
matrix of size $k\times k$ is written as $I_k$ or simply as $I$ if $k$ is
understood. The involution $\sigma$ determines the
\emph{$\sigma$-transposition}
\begin{equation*}
    \Sigma : R\to R : M=(m_{ij})\mapsto M^\Sigma:=(m_{ji}^\sigma).
\end{equation*}
It is an antiautomorphism of $R$. The elements of $ H_\sigma:=\{X\in R \mid
X=X^\Sigma\}$ are the \emph{$\sigma$-Hermitian matrices\/} of $R$. If $M\in R$
is invertible then $M^{-\Sigma}$ is used as a shorthand for
$(M^{-1})^{\Sigma}=(M^{\Sigma})^{-1}$. In the special case that $\sigma=\id_K$
the field $K$ is commutative, and we obtain the subset of \emph{symmetric
matrices\/} of $K^{n\times n}$.

\par
The set $H_\sigma$ need not be closed under matrix multiplication. In the
terminology of \cite[3.1.5]{blunck+he-05} $H_\sigma$ is a \emph{Jordan
system\/} of $R$, where $R=K^{n\times n}$ is considered as an algebra over the
centre $Z(K)$ of $K$. This means that $H_\sigma$ is a subspace of the
$Z(K)$-vector space $R$ which contains $I$, and which has the property that
\begin{equation}\label{eq:(i)}
    A^{-1}\in H_\sigma \mbox{~~~for all~~~} A\in \GL_n(K)\cap H_\sigma.
\end{equation}
Moreover, $ H_\sigma$ is \emph{Jordan closed}, i.~e, it satisfies the condition
\begin{equation}\label{eq:(ii)}
    ABA\in H_\sigma \mbox{~~~for all~~~} A,B\in  H_\sigma.
\end{equation}
In \cite[Lemma~3.1.11]{blunck+he-05} is is shown that condition (\ref{eq:(ii)})
follows from (\ref{eq:(i)}) under a certain richness assumption on $ H_\sigma$
(called \emph{strongness\/} there). See also \cite{herz-92}.
\par
Generalising the definition in \cite[3.1.14]{blunck+he-05} we define the
\emph{projective line\/} over $ H_\sigma$ (irrespective of whether $H_\sigma$
is strong or not) by
\begin{equation}\label{eq:def}
    \bP( H_\sigma) = \{R(T_2T_1 - I, T_2)\mid T_1,T_2\in  H_\sigma\}.
\end{equation}
Note that this definition makes use of multiplication in the ambient matrix
ring $R=K^{n\times n}$ and that $\bP( H_\sigma)$ is a subset of the projective
line over the ring $R$. Since we do not adopt an assumption on the strongness
of $ H_\sigma$, we cannot apply any results from \cite{blunck+he-05}.
\par
We now recall the definition of the \emph{projective space of
$\sigma$-Hermitian matrices}. Following \cite[III~\S~3]{dieu-71} and
\cite[6.8]{wan-96} we consider the left vector space $K^{2n}$ and the
non-degenerate $\sigma$-anti-Hermitian sesquilinear form $\beta:K^{2n}\times
K^{2n}\to K$ given (with respect to the standard basis) by the matrix
\begin{equation}\label{eq:beta}
    \begin{pmatrix}0 & I_n\\-I_n & 0\end{pmatrix}\in\GL_{2n}(K).
\end{equation}
The basic notions and results about sesquilinear forms which will be used below
without further reference can be found in \cite[664--666]{cohen-95}. See also
\cite[I~\S 6--11]{dieu-71} and \cite[\S 8]{tits-74a}. For all
$x=(x_1,x_2,\ldots,x_{2n})\in K^{2n}$ we obtain
\begin{equation}\label{eq:trace}
    (x,x)^\beta = w-w^\sigma
    \mbox{~~~with~~~}
    w:=\sum_{i=1}^n \left( x_i x_{i+n}^\sigma \right)
\end{equation}
or, in other words, $\beta$ is \emph{trace-valued}. We read off from the upper
left corner of the matrix in (\ref{eq:beta}) that the span of the first $n$
vectors of the standard basis is totally isotropic (with respect to $\beta$).
Therefore all maximal totally isotropic subspaces have dimension $n$. The set
comprising all maximal totally isotropic subspaces is the point set of the
\emph{projective space of $\sigma$-Hermitian matrices\/} or, in another
terminology, the point set of the \emph{dual polar space\/} \cite{cameron-82a}
given by $\beta$. Suppose that $(A,B)\in R^2$ satisfies $\rank (A,B)=n$. Then
the ($n$-dimensional) row space of $(A,B)\in K^{n\times 2n}$ is totally
isotropic if, and only if,
\begin{equation}\label{eq:AB=BA}
    AB^\Sigma = BA^\Sigma.
\end{equation}
This is immediate by multiplying the matrix in (\ref{eq:beta}) by $(A,B)$ from
the left and $(A,B)^\Sigma$ from the right hand side; see
\cite[Proposition~6.41]{wan-96}.
\par
In terms of our convention from Section~\ref{se:square} our main result is as
follows:

\begin{thm}\label{thm:1}
Let $K$ be any field admitting an involution $\sigma$. The point set of the
projective space of $\sigma$-Hermitian $n\times n$ matrices over $K$ coincides
with the projective line over the Jordan system $H_\sigma$ of all
$\sigma$-Hermitian matrices of $R=K^{n\times n}$.
\end{thm}

We postpone the proof until we have established two lemmas. We note that
Lemma~\ref{lem:2} generalises \cite[Satz~10.2.3]{blunck+he-05}, where $\beta$
is assumed to be an alternating bilinear form, i.~e., $K$ is commutative and
$\sigma=\id_K$. The proof from there makes use of the fact that all
one-dimensional subspaces of $K^{2n}$ are totally isotropic, but this property
does not hold in general. Therefore our proof follows another strategy.

\begin{lem}\label{lem:1}
Let $U=V\oplus W$ be a maximal totally isotropic subspace which is given as
direct sum of subspaces $V$ and $W$. Then there exists a maximal totally
isotropic subspace, say $X$, such that $X\cap V^\perp=W$.
\end{lem}

\begin{proof}
Let $\dim V=k$. Due to $\dim W=n-k$, $\dim V^\perp = 2n-k$, and
$U=U^\perp\subset V^\perp$, there exists a basis $(b_1,b_2,\ldots,b_{2n})$ of
$K^{2n}$ such that
\begin{eqnarray}\label{eq:VWV}
  V       &=& \spn(b_1,b_2,\ldots,b_{k}), \nonumber\\
  W       &=& \spn(b_{k+1},b_{k+2},\ldots,b_{n}), \\
  V^\perp &=& U\oplus \spn(b_{n+k+1},b_{n+k+2},\ldots,b_{2n}).\nonumber
\end{eqnarray}
The remaining basis vectors $b_{n+1},b_{n+2},\ldots, b_{n+k}$ can be chosen
arbitrarily. The matrix of $\beta$ with respect to $(b_i)$ can be written in
block form as
\begin{equation}\label{eq:M.alt}\renewcommand\arraystretch{1.2}
    M=\left(\!\!
        \begin{array}{cccc}
        0 & 0 & A & 0 \\
        \cline{2-3}
        0         &\multicolumn{1}{|c}{0}        &\multicolumn{1}{c|}{B}          & C \\
        -A^\Sigma &\multicolumn{1}{|c}{-B^\Sigma}&\multicolumn{1}{c|}{D-D^\Sigma} & E \\
        \cline{2-3}
        0 & -C^\Sigma & -E^\Sigma &F-F^\Sigma
    \end{array}\!\!\right)
\end{equation}
with $A\in\GL_k(K)$ and $C\in \GL_{n-k}(K)$. We remark that appropriate
matrices $D\in K^{k\times k}$ and $F\in K^{(n-k)\times (n-k)}$ exist because of
(\ref{eq:trace}). Our aim is to go over to a new basis as follows: All basis
vectors in $V$, $W$, and $V^\perp$ will stay unchanged. The remaining basis
vectors $b_{n+1},b_{n+2},\ldots, b_{n+k}$ will be replaced in such a way that
all entries in the highlighted submatrix turn to zero when performing the
associate transformation on $M$.
\par
This task can easily be accomplished in terms of several elementary row and
column transformations: First one adds appropriate linear combinations of the
last $n-k$ rows of $M$ to the $k$ rows of the third horizontal block in order
to eliminate $-B^\Sigma$. This is possible, since $-C^\Sigma\in\GL_{n-k}(K)$.
Subsequently, the corresponding column transformations will eliminate $B$. Now
the $(3,3)$-block of the transformed matrix reads $D-D^\Sigma+(*)-(*)^\Sigma$.
Next, the first $k$ rows are used to eliminate $D+(*)$. This can be carried
out, due to $A\in\GL_{k}(K)$. Finally, one applies the corresponding column
operations. Altogether the transition from the basis $(b_i)$ to the new basis
$(b_j')$ is given by the matrix
\begin{equation*}\label{eq:T}
        T:=\begin{pmatrix}
         I_k & 0 & 0 & 0\\
         0 & I_{n-k} & 0 & 0 \\
         \left((E-B^\Sigma C^{-\Sigma}F)C^{-1}B-D\right)A^{-1} & 0 & \;\;I_{k}\;\; & -B^\Sigma C^{-\Sigma}\\
         0 & 0& 0& I_{n-k}
    \end{pmatrix}\in\GL_{2n}(K).
\end{equation*}
The elimination from above can be summarised as
\begin{equation}\label{eq:M.neu}\renewcommand\arraystretch{1.2}
    T\cdot M \cdot T^\Sigma =
    \begin{pmatrix}
         0& 0 & A & 0\\
         \cline{2-3}
         0 & \multicolumn{1}{|c}{0} & \multicolumn{1}{c|}{0} & C \\
         -A^\Sigma & \multicolumn{1}{|c}{0} & \multicolumn{1}{c|}{0} & E-B^\Sigma C^{-\Sigma}(F-F^\Sigma)\\
         \cline{2-3}
         0 & -C^\Sigma& -E^\Sigma+(F^\Sigma-F)C^{-1}B& F-F^\Sigma
    \end{pmatrix}
\end{equation}
and gives the new matrix for $\beta$. By our construction and due to the form
of the matrix in (\ref{eq:M.neu}), the basis vectors
$b_{k+1}',b_{k+2}',\ldots,b_{n+k}'$ generate a subspace $X$ with the required
properties.
\end{proof}

\begin{lem}\label{lem:2}
Let $U_1$ and $U_2$ be two maximal totally isotropic subspaces of
$(K^{2n},\beta)$. Then there exists a maximal totally isotropic subspace $X$
which is a common complement of $U_1$ and $U_2$.
\end{lem}

\begin{proof}
(a) Let $V:=U_1\cap U_2$ and put $k:=\dim V$. Choose subspaces $W_i$ such that
$U_i=V\oplus W_i$ for $i\in\{1,2\}$. Then $V^\perp=V\oplus W_1\oplus W_2$. The
restriction of $\beta$ to $V^\perp\times V^\perp$ might be degenerate, with
$V=V^{\perp\perp}$ being the radical of the restricted form. Consequently, the
restriction of $\beta$ to $(W_1\oplus W_2)\times (W_1\oplus W_2)$ is
non-degenerate. It will be written as $\beta_{12}$. There exist bases
$(b_1,b_2,\ldots,b_{n-k})$ and $(b_{n-k+1},b_{n-k+2},\ldots,b_{2n-2k})$ of
$W_1$ and $W_2$, respectively. The matrix of $\beta_{12}$ with respect to
$(b_{1},b_{2},\ldots,b_{2n-2k})$ has the form
\begin{equation}\label{eq:U1}
    \begin{pmatrix}
        0 & A\\-A^\Sigma&0
    \end{pmatrix}
    \mbox{~~~with~~~}A\in\GL_{n-k}(K).
\end{equation}
Let $A^{-1}$ be the matrix describing the change from the basis
$(b_{1},b_{2},\ldots,b_{n-k})$ of $W_1$ to a new basis
$(b_{1}',b_{2}',\ldots,b_{n-k}')$, say. Thus the matrix
\begin{equation*}
    \begin{pmatrix}
        A^{-1} & 0\\0&I
    \end{pmatrix}
    \begin{pmatrix}
        0 & A\\-A^\Sigma&0
    \end{pmatrix}
        \begin{pmatrix}
        A^{-1} & 0\\0&I
    \end{pmatrix}^\Sigma =
        \begin{pmatrix}
        0 & I\\-I&0
    \end{pmatrix}
\end{equation*}
describes $\beta_{12}$ with respect to the basis
$(b_1',\ldots,b_{n-k}',b_{n-k+1},\ldots,b_{2n-2k})$. Using this new matrix for
$\beta_{12}$ it is straightforward to show that
\begin{equation*}
    (b_r'+b_{n+r}, b_s'+b_{n+s})^\beta=0 \mbox{~~~for all~~~} r,s\in \{1,2\ldots,n-k\}.
\end{equation*}
Hence
\begin{equation*}
    W:=\spn (b_1'+b_{n+1}, b_2'+b_{n+2},\ldots,b_{n-k}'+b_{2n-k})
\end{equation*}
is a totally isotropic subspace. Furthermore, we have $\dim W= n-k$ and
$W_1\oplus W_2 = W_1\oplus W = W_2\oplus W$.
\par
(b) Let $U:=V\oplus W$, the sum being direct due to $V\cap (W_1\oplus W_2)=0$.
So $\dim U=n$. From $W\subset V^\perp\cap W^\perp$ follows $W^\perp\supset
V\oplus W=U$, whereas $V^\perp \supset V \oplus W = U$ is obvious. Therefore
$U\subset V^\perp\cap W^\perp = (V\oplus W)^\perp = U^\perp$. Summing up, we
have proved that $U$ is a maximal totally isotropic subspace of $K^{2n}$.
\par
By Lemma~\ref{lem:1}, applied to $U=V\oplus W$, there exists a maximal totally
isotropic subspace $X$ with $X\cap V^\perp=W$. Consequently, $X\cap U_i= (X\cap
V^\perp) \cap U_i= W\cap U_i=0$ which in turn shows that $X$ is a common
complement of $U_1$ and $U_2$.
\end{proof}

We are now in a position to give the promised proof of Theorem~\ref{thm:1}.

\begin{proof}
(a) Any point of the projective line over $ H_\sigma$ can be written in the
form $R(T_2T_1 - I, T_2)$ with $T_1,T_2\in H_\sigma$ according to
(\ref{eq:def}). Then
\begin{equation*}
    (T_2T_1 - I) T_2^\Sigma = T_2T_1T_2 - T_2 =  T_2(T_2T_1 - I)^\Sigma.
\end{equation*}
Now (\ref{eq:AB=BA}) shows that the left row space of $(T_2T_1 - I, T_2)$ is a
maximal totally isotropic subspace.
\par
(b) Let the left row space of $(A,B)$ be a maximal totally isotropic subspace.
We consider the maximal totally isotropic subspace given as left row space of
the matrix $(I,0)$. By Lemma~\ref{lem:2} there exists a maximal totally
isotropic subspace of $K^{2n}$ which is a common complement. In matrix form it
can be written as $(C,D)$. So in terms of $\bP(R)$ we have $R(I,0)\dis
R(C,D)\dis R(A,B)$ or, said differently, each of the matrices
\begin{equation*}
    \begin{pmatrix}I&0\\C&D\end{pmatrix},\;\;
    \begin{pmatrix}C&D\\A& B\end{pmatrix}
\end{equation*}
is invertible. We may thus put $D=I$ without loss of generality. Clearly,
\begin{equation*}
    \begin{pmatrix}C&I\\A-BC& 0\end{pmatrix}=
    \begin{pmatrix}I&0\\-B&I\end{pmatrix}
    \begin{pmatrix}C&I\\A& B\end{pmatrix}
    \in\GL_2(R)=\GL_{2n}(K),
\end{equation*}
whence $A-BC\in\GL_n(K)$. Defining $T_1:=C$ and $T_2:=(BC-A)^{-1}B$ gives
\begin{eqnarray*}
    R(T_2T_1 - I,T_2) &=& R\big((BC-A)^{-1}BC - I, (BC-A)^{-1}B\big)\\
                      &=& R\big(BC - (BC-A)I, B\big)\\
                      &=& R(A,B).
\end{eqnarray*}
Since the left row space of $(C,I)$ is totally isotropic, we have
$T_1=C=C^\Sigma\in H_\sigma$ by (\ref{eq:AB=BA}). Applying (\ref{eq:AB=BA}) to
the totally isotropic left row space of $(T_2T_1 - I,T_2)$ therefore gives
\begin{equation}\label{}
    T_2T_1T_2^\Sigma - T_2^\Sigma = T_2T_1 T_2^\Sigma - T_2,
\end{equation}
whence $T_2\in H_\sigma$. This completes the proof.
\end{proof}
\begin{rem}\label{rem:5}
In Remarks~\ref{rem:1}--\ref{rem:4} we sketched several applications of
Bartolone's parametrisation to the geometry of square matrices. In view of
Theorem~\ref{thm:1} it is now a straightforward task to carry them over,
\emph{mutatis mutandis}, to the geometry of $\sigma$-Hermitian matrices. For
example, in the projective space of $\sigma$-Hermitian matrices we can
parametrise any maximal set of mutually adjacent elements containing the left
row space of $(I,0)$ via matrices of the form $(T_2T_1-I,T_2)$ as follows:
$T_1\in H_\sigma$ is arbitrary, whereas $T_2:=c^{(0)}{}^{\Trans}\cdot t\cdot
c^{(0)}$ for a fixed vector $c^{(0)}\in K^{2n}\setminus\{0\}$ and a variable
$t\in K$ satisfying $t=t^{\sigma}$.
\end{rem}
In contrast to this analogy the following difference has to be pointed out: It
was shown in \cite[Section~4]{kwiat+p-09a} that the characterisation of
adjacency in $\Gr_{2n,n}(K)$ in terms of the distant relation from
\cite[Theorem~3.2]{blunck+h-05a} \emph{cannot be carried over literally\/} to a
projective space of symmetric matrices over a commutative field of
characteristic $2$. So the following problem arises: Is it possible to express
the adjacency relation on any projective space of $\sigma$-Hermitian matrices
in terms of the distant relation on $\bP(H_\sigma)$? An affirmative answer
would imply that the distant preserving bijections of $\bP(H_\sigma)$ are
precisely the adjacency preserving bijections of the projective matrix space
over $H_\sigma$.

\subsubsection*{Acknowledgement}

This work was carried out within the framework of the Cooperation Group
``Finite Projective Ring Geometries: An Intriguing Emerging Link Between
Quantum Information Theory, Black-Hole Physics, and Chemistry of Coupling'' at
the Center for Interdisciplinary Research (ZiF), University of Bielefeld,
Germany.


\begin{thebibliography}{10}\itemsep0mm

\bibitem{bart-89} C.~Bartolone.
\newblock Jordan homomorphisms, chain geometries and the fundamental theorem.
\newblock {\em Abh.\ Math.\ Sem.\ Univ.\ Hamburg}, 59:93--99, 1989.

\bibitem{blunck-99} A.~Blunck.
\newblock Regular spreads and chain geometries.
\newblock {\em Bull.\ Belg.\ Math.~Soc.\ Simon Stevin}, 6:589--603, 1999.

\bibitem{blunck+h-00b} A.~Blunck and H.~Havlicek.
\newblock Projective representations {I}.\ {P}rojective lines over rings.
\newblock {\em Abh.\ Math.\ Sem.\ Univ.\ Hamburg}, 70:287--299, 2000.

\bibitem{blunck+h-03} A.~Blunck and H.~Havlicek.
\newblock Jordan homomorphisms and harmonic mappings.
\newblock {\em Monatsh.\ Math.}, 139:111--127, 2003.

\bibitem{blunck+h-05a} A.~Blunck and H.~Havlicek.
\newblock On bijections that preserve complementarity of subspaces.
\newblock {\em Discrete Math.}, 301:46--56, 2005.

\bibitem{blunck+he-05} A.~Blunck and A.~Herzer.
\newblock {\em Kettengeometrien -- {E}ine {E}inf\"{u}hrung}.
\newblock Shaker Verlag, Aachen, 2005.

\bibitem{cameron-82a} P.~J.~Cameron.
\newblock Dual polar spaces.
\newblock {\em Geom. Dedicata}, 12(1):75--85, 1982.

\bibitem{cohen-95} A.~M.~Cohen.
\newblock Point-line spaces related to buildings.
\newblock In F.~Buekenhout, editor, {\em Handbook of Incidence Geometry},
  pages 647--737. Elsevier, Amsterdam, 1995.

\bibitem{dieu-71} J.~A.~Dieudonn\'e.
\newblock {\em La G\'eom\'etrie des Groupes Classiques}.
\newblock Springer, Berlin Heidelberg New York, 3rd edition, 1971.

\bibitem{herz-92} A.~Herzer.
\newblock On sets of subspaces closed under reguli.
\newblock {\em Geom.\ Dedicata}, 41:89--99, 1992.

\bibitem{herz-95} A.~Herzer.
\newblock Chain geometries.
\newblock In F.~Buekenhout, editor, {\em Handbook of Incidence Geometry},
  pages 781--842. Elsevier, Amsterdam, 1995.

\bibitem{kwiat+p-09a} M.~Kwiatkowski and M.~Pankov.
\newblock Opposite relation on dual polar spaces and half-spin {G}rassmann
  spaces.
\newblock {\em Results Math.}, 54(3-4):301--308, 2009.

\bibitem{thas-71} J.~A.~Thas.
\newblock The $m$-dimensional projective space ${S_m(M_n(GF(q)))}$ over the
  total matrix algebra ${M_n(GF(q))}$ of the $n\times n$-matrices with elements
  in the {G}alois field ${GF(q)}$.
\newblock {\em Rend.\ Mat.\ Roma (6)}, 4:459--532, 1971.

\bibitem{tits-74a} J.~Tits.
\newblock {\em Buildings of spherical type and finite {BN}-pairs}.
\newblock Lecture Notes in Mathematics, Vol. 386. Springer-Verlag, Berlin,
  1974.

\bibitem{veld-85} F.~D.~Veldkamp.
\newblock Projective ring planes and their homomorphisms.
\newblock In R.~Kaya, P.~Plaumann, and K.~Strambach, editors, {\em Rings and
  Geometry}, pages 289--350. D.\ Reidel, Dordrecht, 1985.

\bibitem{wan-96} Z.-X.~Wan.
\newblock {\em Geometry of Matrices}.
\newblock World Scientific, Singapore, 1996.

\end{thebibliography}

\noindent Andrea Blunck\\
Department Mathematik\\
Universit\"{a}t Hamburg\\
Bundesstra{\ss}e 55\\
D-20146 Hamburg\\
Germany\\
\texttt{andrea.blunck@math.uni-hamburg.de}
\\~\\
\noindent
Hans Havlicek\\
Institut f\"{u}r Diskrete Mathematik und Geometrie\\
Technische Universit\"{a}t\\
Wiedner Hauptstra{\ss}e 8--10/104\\
A-1040 Wien\\
Austria\\
\texttt{havlicek@geometrie.tuwien.ac.at}


\end{document}